\documentclass[11pt]{amsart}
\usepackage{amsthm,amsmath,amssymb}
\usepackage{hyperref}

{\theoremstyle{definition}
    \newtheorem{theorem}{Theorem}[section]
    \newtheorem{proposition}[theorem]{Proposition}
    
    \newtheorem{corollary}[theorem]{Corollary}
    \newtheorem{definition}[theorem]{Definition}
    
    \newtheorem{example}[theorem]{Example}
    \newtheorem{parrafo}[theorem]{{\!}}
}
\numberwithin{equation}{theorem}

\newcommand{\OO}{{\mathcal {O}}}

\DeclareMathOperator{\Coeff}{Coeff}

\DeclareMathOperator{\Diff}{Diff}
\DeclareMathOperator{\Fc}{{\digamma\! c}}

\DeclareMathOperator{\GrQ}{Gr-\mathbb{Q}}

\DeclareMathOperator{\QPlus}{\mathbb{Q}_{\geq 0}}
\DeclareMathOperator{\QPos}{\mathbb{Q}_{> 0}}
\DeclareMathOperator{\Sing}{Sing}
\DeclareMathOperator{\MaxB}{\underline{\mathbf{Max}}}
\DeclareMathOperator{\ord}{ord}
\DeclareMathOperator{\RPlus}{\mathbb{R}_{\geq 0}}
\DeclareMathOperator{\Spec}{Spec}

\def\dern#1#2#3{\dfrac{\partial^{#3} #1}{\partial {#2}^{#3}}}

\title{Coefficient and elimination algebras in Resolution of Singularities.}

\author{Roc\'{\i}o Blanco}
\author{Santiago Encinas}
\thanks{Partially supported by MTM2009-07291}

\address{Dep. Matem\'aticas.
Universidad de Castilla la Mancha.
Avenida de los Alfares 42.
16071 Cuenca. Spain.}

\address{Dep. Matem\'atica Aplicada.
Universidad de Valladolid.
Avda. Salamanca s/n.
47014 Valladolid. Spain}

\date{\today}
\dedicatory{ Dedicated to H. Hironaka in his 80th birthday.}
\begin{document}

\maketitle

\section*{Introduction.}

Given a variety \( X \) over a field \( k \) one wants to find a
desingularization, which is a proper and birational morphism \( 
X'\to X \), where \( X' \) is a regular variety and the morphism is 
an isomorphism over the regular points of \( X \).

If \( X \) is embedded in a regular variety \( W \), there is a 
notion of embedded desingularization and related to this is the 
notion of log-resolution of ideals in \( \mathcal{O}_W \).

When the field \( k \) has characteristic zero it is well known that 
the problem of resolution is solved. The first proof of the existence 
of resolution of singularities is due to H. Hironaka in his 
monumental work \cite{Hironaka1964} (see also \cite{Hironaka1977}).

If characteristic of \( k \) is positive the problem of resolution in 
arbitrary dimension is still open. See \cite{Hauser2010} for recent 
advances and obstructions (see also \cite{Hauser2003}).

The proof by Hironaka is existential. There 
are constructive proofs, always in characteristic zero case, see for 
instance \cite{Villamayor1989}, \cite{Villamayor1992}, 
\cite{BierstoneMilman1997}, we refer to \cite{Hauser2003} for a 
complete list of references.
Those constructive proofs give rise to algorithmic resolution of singularities, 
that allows to perform implementation at the computer 
\cite{BodnarSchicho2000b}, \cite{FruhbisPfister2004}.

Recently some techniques have appeared in order to try to prove the 
problem of resolution of singularities in the positive characteristic 
case. Rees algebras seem to be a useful tool in this context.
Hironaka in \cite{Hironaka2003} and \cite{Hironaka2005} propose to 
use Rees algebras for proving log-resolution of ideals. The advantage 
of Rees algebras is that the algebra encodes in one object many 
ideals which are ``equivalent'' for the problem of log-resolution.
Also Rees algebras have a good behavior with respect to integral 
closure, see for instance \cite{Villamayor2008} and 
\cite{Villamayor2007}.
On the other hand Kawanoue and Matsuki \cite{Kawanoue2007}, 
\cite{KawanoueMatsuki2006prep} use a different object, called 
idealistic filtration. Which is similar to Rees algebras but with a 
grading over the real numbers.

In this paper we compare those structures and construct \( \mathbb{Q} 
\)-Rees algebras (\ref{DefQReesAlg}), which are algebras with grading 
over the rational numbers.
We will see that Rees algebras, idealistic filtrations and \( 
\mathbb{Q} \)-Rees algebras encode (up to integral closure) the same 
information (\ref{TresEquiv}).
One motivation to extend Rees algebras to a \( \mathbb{Q} \)-grading 
comes from the scaling operation (\ref{DefScaling}), which is needed 
in the process of resolution of singularities.
Since we restrict to rational numbers all properties related to 
integral closure and finiteness come easily, see \ref{GrQlimit}.

We will use \( \mathbb{Q} \)-Rees algebras in order to construct 
log-resolution of ideals in the characteristic zero case.
Along the paper we have specified which constructions and results are 
valid in general or only in characteristic zero.

Section~\ref{SeccQRA} is devoted to introduce the several algebras we 
have mentioned and to see that they are ``equivalent'' up to integral 
closure.
In section~\ref{SeccSheaves} we extend the notion of \( \mathbb{Q} 
\)-Rees algebra to sheaves and define the order function.

In section~\ref{SeccOperations} we define all operations on algebras: 
integral closure, radical saturation, differential saturation, 
scaling and the division by a hypersurface.
This last operation together with the function \( \ord \), defined for 
any \( \mathbb{Q} \)-Rees algebra, are a key point in order to 
construct a resolution algorithm.
Except integral closure, all operations may be expressed easily in 
terms of generators.

In section~\ref{SeccCoeff} we define coefficient algebra and see its 
properties. The stability by monoidal transformations 
\ref{SeqTransIntersec} is valid in general, but the ``maximal 
contact'' (\ref{ThContMaxLocal}) holds only if characteristic of \( 
k \) is zero.
We also prove, in this setting, theorem~\ref{ThJarek} which was 
proved first by W{\l}odarczyk \cite{Wlodarczyk2005}. This theorem 
ensures that coefficient algebras do not depend on the choice of 
``maximal contact''.

Section~\ref{SeccElimin} defines elimination algebras in the most
direct way, just to allow computations.  See \cite{Villamayor2007} for
a more detailed discussion on elimination algebras and the behavior 
with integral closure.
Elimination algebras fail to have a good stability by monoidal 
transformations, as illustrates example~\ref{EjemElimina}.
On the other hand, if characteristic of \( k \) is zero, then both 
coefficient algebra and elimination algebra are isomorphic 
\ref{CoeffIsoElim} for the \'etale topology.

Finally in section~\ref{SeccAlgorithm} we construct an algorithm of 
resolution of \( \mathbb{Q} \)-Rees algebras which gives a 
log-resolution algorithm.
We see in \ref{EquivSameResol} that we construct the same
resolution for algebras with the same integral closure, see also 
\cite{EncinasVillamayor2007}.
\medskip

The authors want to thank H. Kawanoue for useful suggestions.
We also want to thank to O. Villamayor, A. Bravo, A. Benito,
M. L. Garc\'{\i}a Escamilla, J. Schicho,
H. Hauser and D. Panazzolo for several mathematical
conversations.

\section{$\mathbb{Q}$-Rees algebras.} \label{SeccQRA}

Fix \( R \) to be a noetherian ring over a field \( k \). The 
characteristic of the field \( k \)
is arbitrary, unless specified.

We consider Rees algebras and extend this notion to a suitable
definition of Rees algebras over rational numbers.

\begin{definition} \label{DefReesAlg}
\cite{Villamayor2008}
Let \( R \) be a noetherian ring. Consider the graded algebra:
\begin{equation*}
    R[T]=\bigoplus_{n\in\mathbb{N}}RT^{n}
\end{equation*}
A Rees algebra in \( R \) is a graded subalgebra
\begin{equation*}
    \mathcal{J}=\bigoplus_{n\in\mathbb{N}}J_{n}T^{n}\subset R[T]
\end{equation*}
such that \( J_{0}=R \) and \( \mathcal{J} \) is finitely generated as
\( R \)-algebra.
\end{definition}

Equivalently we may define a Rees algebra as a collection of ideals \( 
\{J_{n}\}_{n\in\mathbb{N}} \) such that:
\begin{enumerate}
    \item  \( J_{0}=R \),

    \item  \( J_{m}J_{n}\subset J_{m+n} \) for any \( m,n\in\mathbb{N} \),
    and

    \item there exist elements \( f_{1},\ldots,f_{r}\in R \) and
    degrees \( n_{1},\ldots,n_{r}\in\mathbb{N} \), with \(
    f_{i}\in J_{n_{i}} \), \( i=1,\ldots,r \), such that for any \( 
    n\in\mathbb{N} \) the ideal
    \( J_{n} \) is generated (as ideal) by the set
    \begin{equation*}
        \{f_{i_{1}}\cdots f_{i_{\ell}}\mid 
	n_{i_{1}}+\cdots+n_{i_{\ell}}=n\}
    \end{equation*}
\end{enumerate}

We will be interested in considering the equivalence class of Rees 
algebras up to integral closure.
\begin{definition} \label{DefIntEquivRA}
Two Rees algebras (\ref{DefReesAlg}) \(
\mathcal{J}_{1},\mathcal{J}_{2}\subset R[T] \) are \emph{equivalent}
if they have the same integral closure as subrings of \( R[T] \).
\end{definition}

\begin{parrafo} \label{ExtraPropRees}
Let \( \mathcal{J}=\oplus_{n}J_{n}T^{n} \) be a Rees algebra.
For any \( n\in\mathbb{N} \) set the ideal
\begin{equation*}
    I_{n}=\sum_{m\geq n}J_{m}
\end{equation*}
It can be checked that \( \mathcal{I}=\oplus_{n}I_{n}T^{n} \) is a 
Rees algebra. In fact if \( \{f_{1}T^{n_{1}},\ldots,f_{r}T^{n_{r}}\} 
\) is a set of generators of the algebra \( \mathcal{J} \) then
\begin{equation*}
    \{f_{i}T^{m}\mid i=1,\ldots,r,\ 0\leq m\leq n_{i}\}
\end{equation*}
is a set of generators of \( \mathcal{I} \).
In fact it can be checked that for any \( n\in\mathbb{N} \)
\begin{equation*}
    I_{n}=\left\langle f_{i_{1}}\cdots f_{i_{\ell}} \mid
    n_{i_{1}}+\cdots +n_{i_{\ell}}\geq n\right\rangle
\end{equation*}

Note that the Rees algebra \( \mathcal{I} \) satisfies the following 
property 
\begin{equation} \label{ExtraPropNRA}
    \text{if } n_{1}\geq n_{2} \text{ then }  I_{n_{1}}\subset I_{n_{2}}
\end{equation}
In fact \( \mathcal{I} \) is the smallest Rees algebra satisfying 
(\ref{ExtraPropNRA}) and such that \( \mathcal{J}\subset\mathcal{I} \).
On the other hand we have that \( \mathcal{J}\subset\mathcal{I} \) is 
a finite extension (i.e. the algebras \( \mathcal{J} \) and \( 
\mathcal{I} \) are equivalent (\ref{DefIntEquivRA})).
In order to prove the last sentence, it is enough to prove that if \( 
f\in J_{n} \) and \( m\leq n \), then \( fT^{m} \) is integral over \( \mathcal{J} \).
We have that
\begin{equation*}
    f\in J_{n} \Longrightarrow f^{m}\in J_{nm} \Longrightarrow 
    f^{n}\in J_{nm}
\end{equation*}
so that the element \( fT^{m} \) fulfills the monic equation \( 
Z^{n}-(f^{n}T^{nm})=0 \).
\end{parrafo}

So that, up to integral closure, condition (\ref{ExtraPropNRA}) may 
be added to our definition of algebras.
\begin{definition} \label{DefNReesAlg}
An \( \mathbb{N} \)-Rees algebra is a Rees algebra \( \mathcal{J} \)
(\ref{DefReesAlg}) satisfying (\ref{ExtraPropNRA}).  Equivalently, an
\( \mathbb{N} \)-Rees algebra is a collection of ideals \(
\{J_{n}\}_{n\in\mathbb{N}} \) such that
\begin{enumerate}
    \item  \( J_{0}=R \),

    \item  \( J_{m}J_{n}\subset J_{mn} \), for any \( m,n\in\mathbb{N}
    \),

    \item  if \( n\leq m \) then \( J_{m}\subset J_{n} \), and

    \item  there exist elements \( f_{1},\ldots,f_{r} \) and degrees \( 
    n_{1},\ldots,n_{r} \), with \(
    f_{i}\in J_{n_{i}} \), \( i=1,\ldots,r \), such that for any \( 
    n\in\mathbb{N} \) the ideal
    \( J_{n} \) is generated (as ideal) by the set
    \begin{equation*}
        \{f_{i_{1}}\cdots f_{i_{\ell}}\mid 
	n_{i_{1}}+\cdots+n_{i_{\ell}}\geq n\}
    \end{equation*}
\end{enumerate}
\end{definition}

\begin{parrafo}
If \( R \) is a regular local ring there is a well-defined notion of order 
of ideals \( J\subset R \), which can be extended to \( \mathbb{N}
\)-Rees algebras (see \ref{DefOrdRA}). In general the order will be a 
rational number and we will be interested in algebras with order one.
Algebras with order one will allow to define inductive procedures in
the process of resolution of singularities.
So that we will need to \emph{normalize} algebras in order to have
order one, this procedure will be called \emph{scaling} and it will
lead naturally to consider graduation over the rational numbers.
\end{parrafo}

\begin{parrafo} \label{GrQlimit}
Set \( \QPlus=\{a\in\mathbb{Q}\mid a\geq 0\} \).
We will consider the \( \mathbb{Q} \)-graded algebra
\begin{equation*}
    \GrQ(R)=\bigoplus_{a\in\QPlus} R T^{a}
\end{equation*}
Note that this graded algebra is the limit of the \( \mathbb{N}
\)-graded algebras \( R[T^{\frac{1}{N}}] \) for \( N\in\mathbb{N} \)
\begin{equation*}
    \GrQ(R)=\bigcup_{N\in\mathbb{N}}R[T^{\frac{1}{N}}]
\end{equation*}
So that every finite set of \( \GrQ(R) \) is included in \(
R[T^{\frac{1}{N}}] \) for \( N \) big enough.
This condition will allow us to use properties of Rees algebras for
\emph{finitely generated} sub-algebras of \( \GrQ(R) \).
\end{parrafo}

\begin{definition} \label{DefQReesAlg}
A \( \mathbb{Q} \)-Rees algebra over \( R \) is a graded subalgebra 
\begin{equation*}
    \mathcal{J}=\bigoplus_{a\in\QPlus} J_{a}T^{a}\subset \GrQ(R)
\end{equation*}
such that the collection of ideals \( \{J_{a}\}_{a\in\QPlus} \)
satisfies the following properties:
\begin{enumerate}
    \item  \( J_{0}=R \),

    \item  \( J_{a}J_{b}\subset J_{a+b} \) for all \( 
    a,b\in\QPlus \),

    \item  \( J_{b}\subset J_{a} \) if \( a\leq b \) and

    \item \( \mathcal{J} \) is finitely generated, which means that
    there are elements \( f_{1},\ldots,f_{r}\in R \) and degrees \(
    a_{1},\ldots,a_{r}\in\QPlus \), with \( f_{i}\in
    J_{a_{i}} \), \( i=1,\ldots,r \), such that for any \( a\in\QPlus
    \) the ideal \( J_{a} \) is generated (as ideal) by the set
    \begin{equation*}
	\{f_{i_{1}}\cdots f_{i_{\ell}}\mid 
	a_{i_{1}}+\cdots+a_{i_{\ell}}\geq a\}
    \end{equation*}
\end{enumerate}
\end{definition}

\begin{parrafo} \label{QReesAlg_fg}
With the notation of \ref{DefQReesAlg}, we say that the \( \mathbb{Q}
\)-Rees algebra is generated by \( f_{1}T^{a_{1}},\ldots,f_{r}T^{a_{r}} \).

Note that the \( \mathbb{Q} \)-Rees algebra \( \mathcal{J} \) is the smallest 
algebra satisfying properties (1), (2) and (3) of \ref{DefQReesAlg} 
and containing the elements \( f_{1}T^{a_{1}},\ldots,f_{r}T^{a_{r}} \).
\end{parrafo}

\begin{parrafo}
Note that the condition of being finitely generated may be expressed as
follows:
There are homogeneous elements
\( f_{1}T^{a_{1}},\ldots,f_{r}T^{a_{r}}\in \mathcal{J} \)
and there is an integer \( N \) such that
\begin{enumerate}
    \item  \( a_{i}N \) is an integer for \( i=1,\ldots,r \),

    \item  the finite set
    \begin{equation*}
	\{f_{i}T^{\frac{m}{N}}\mid m\leq a_{i}N,\ i=1,\ldots,r\}
    \end{equation*}
    generates \( \mathcal{J}\cap R[T^{\frac{1}{N}}] \) as an \( R 
    \)-algebra, and

    \item for any \( a\in\QPlus \) we have that \( 
    J_{a}=J_{\frac{\lceil aN\rceil}{N}} \).
\end{enumerate}
Where \( \lceil a\rceil \) denotes the smallest integer bigger than or
equal to \( a \).

In other words, a \( \mathbb{Q} \)-Rees algebra \(
\mathcal{J}\subset\GrQ(R) \) is equivalent to an \( \mathbb{N} \)-Rees
algebra in \( R[T^{\frac{1}{N}}] \), for big enough \( N \), where we 
fill rational levels
according to the rule \( J_{a}=J_{\frac{\lceil aN\rceil}{N}} \), \(
a\in\QPlus \).

Note also that if \( \mathcal{J} \) is a \( \mathbb{Q} \)-Rees 
algebra and \( b\in\QPlus \) then \( \mathcal{J}\cap R[T^{b}] \) is an 
\( \mathbb{N} \)-Rees algebra in \( R[T^{b}] \).
\end{parrafo}

Integral closure will be the criterion to be equivalent for \(
\mathbb{Q} \)-Rees algebras.
\begin{definition} \label{DefIntEquivQRA}
Two \( \mathbb{Q} \)-Rees algebras \( \mathcal{J}_{1} \) and \(
\mathcal{J}_{2} \) are equivalent if the Rees algebras \(
\mathcal{J}_{1}\cap R[T] \) and \( \mathcal{J}_{2}\cap R[T] \) have
the same integral closure in \( R[T] \) (i.e. they are equivalent
(\ref{DefIntEquivRA})).
\end{definition}

\begin{proposition}
Let \( \mathcal{J}_{1} \) and \(
\mathcal{J}_{2} \) be two \( \mathbb{Q} \)-Rees algebras in \( \GrQ(R) \).
The following statements are equivalent:
\begin{enumerate}
    \item  \( \mathcal{J}_{1} \) and \( \mathcal{J}_{2} \) are
    equivalent \( \mathbb{Q} \)-Rees algebras (\ref{DefIntEquivQRA}).

    \item For some \( b\in\QPos \), the algebras \( \mathcal{J}_{1}\cap R[T^{b}]
    \) and \( \mathcal{J}_{2}\cap R[T^{b}] \) have the same integral
    closure in \( R[T^{b}] \).

    \item  For any \( b\in\QPos \), the algebras \( \mathcal{J}_{1}\cap R[T^{b}]
    \) and \( \mathcal{J}_{2}\cap R[T^{b}] \) have the same integral
    closure in \( R[T^{b}] \).
\end{enumerate}
\end{proposition}

\begin{proof}
It follows from the fact that all the extensions
\begin{equation*}
    R[T^{M}]\subset R[T], \quad
    R[T]\subset R[T^{\frac{1}{N}}], \quad
    R[T^{\frac{M}{N}}]\subset R[T^{\frac{1}{N}}]
\end{equation*}
are finite for \( M,N\in \mathbb{N} \).
\end{proof}

\begin{parrafo} \label{EquivaleNyQRA}
These algebraic structures, Rees algebras (\ref{DefReesAlg}),
\( \mathbb{N} \)-Rees algebras
(\ref{DefNReesAlg}) and \( \mathbb{Q} \)-Rees algebras
(\ref{DefQReesAlg}) are ``equivalent'' up to integral closure.

First of all, it follows from \ref{ExtraPropRees} that Rees algebras
(\ref{DefReesAlg}) and \( \mathbb{N} \)-Rees algebras
(\ref{DefNReesAlg}) are equivalent up to integral closure.

For the equivalence of \( \mathbb{N} \)-Rees algebras and \( \mathbb{Q} \)-Rees
algebras, we check that we can associate a  \( \mathbb{Q} \)-Rees
algebra to every  \( \mathbb{N} \)-Rees algebra and vice-versa.

Fix a ring \( R \), let \( \mathcal{I}=\oplus_{n} I_{n}T^{n}\subset R[T] \)
be an \( \mathbb{N} \)-Rees algebra.
Set \( \mathcal{I}\GrQ(R)=
\mathcal{J}=\oplus_{a}J_{a}T^{a}\subset \GrQ(R) \) as follows:
\begin{equation*}
    J_{a}=I_{\lceil a\rceil}, \qquad \forall a\in\QPlus
\end{equation*}
It is clear that \( \mathcal{I}\GrQ(R) \) is a \( \mathbb{Q} \)-Rees algebra.

Reciprocally, if \( \mathcal{J} \) is a \( \mathbb{Q} \)-Rees algebra 
then set \( \mathcal{I}=\mathcal{J}\cap R[T] \), which is an \( 
\mathbb{N} \)-Rees algebra.

Now for an \( \mathbb{N} \)-Rees algebra \( \mathcal{I} \), we have 
that \( \mathcal{I}=(\mathcal{I}\GrQ(R))\cap R[T] \).
Reciprocally if \( \mathcal{J} \) is a \( \mathbb{Q} \)-Rees algebra 
then \( \mathcal{J} \) and \( (\mathcal{J}\cap R[T])\GrQ(R) \) 
are equivalent \( \mathbb{Q} \)-Rees algebras (\ref{DefIntEquivQRA}).
\end{parrafo}

Recently Rees algebras have been used in new approaches to resolution 
of singularities in positive characteristic, see \cite{Hironaka2003}, 
\cite{Hironaka2005}, \cite{Villamayor2007}, \cite{Villamayor2008}.
On the other hand, in \cite{Kawanoue2007} and \cite{KawanoueMatsuki2006prep} 
the authors define \emph{idealistic filtrations}, which are 
collections of ideals indexed by real numbers with the same properties
as a \( \mathbb{Q} \)-Rees algebra (\ref{DefQReesAlg}).
\begin{definition}\cite[2.1.3.1]{Kawanoue2007}
An \emph{idealistic filtration} is a collection of ideals \( 
\mathcal{J}=\{J_{a}\}_{a\in\RPlus} \) such that
\begin{enumerate}
    \item  \( J_{0}=R \),
    
    \item \( J_{a}J_{b}\subset J_{a+b} \) for all \( a,b\in\RPlus \) 
    and
    
    \item  \( J_{b}\subset J_{a} \) if \( a\leq b \).
\end{enumerate}
We could also denote \( \mathcal{J}=\oplus_{a\in\RPlus}J_{a}T^{a} \).
\end{definition}

\begin{parrafo} \label{IdFilt_rfg}
The idealistic filtration generated by a subset \( L\subset 
R\times\RPlus \) is the minimal idealistic filtration containing the 
set \( L \).
An idealistic filtration is \emph{rationally and finitely generated} 
(r.f.g) if it is generated by a finite set in \( R\times\QPlus \):
\begin{equation*}
    \{f_{1}T^{a_{1}},\ldots,f_{r}T^{a_{r}}\}, \qquad a_{1},\ldots,a_{r}\in\QPlus
\end{equation*}
\end{parrafo}

\begin{parrafo} \label{TresEquiv}
The three notions: Rees algebras, \( \mathbb{Q} \)-Rees algebras and
rationally and finitely generated idealistic filtrations are
``equivalent'' up to integral closure.

By \ref{EquivaleNyQRA} and \ref{ExtraPropRees} it is enough to see 
the equivalence between \( \mathbb{Q} \)-Rees algebras and rationally 
and finitely generated idealistic filtrations.
And this follows by considering \( \mathbb{Q} \)-Rees algebras and 
rationally and finitely generated idealistic filtrations generated by 
\( f_{1}T^{a_{1}},\ldots,f_{r}T^{a_{r}} \) (\ref{QReesAlg_fg} and 
\ref{IdFilt_rfg}).
\end{parrafo}

\section{Sheaves.} \label{SeccSheaves}

Let \( W \) be a smooth scheme over a (perfect) field \( k \).

\begin{parrafo}
A sequence of coherent sheaves of ideals \( 
\{I_{n}\}_{n\in\mathbb{N}} \), \( I_{n}\subset\OO_{W} \), defines a 
sheaf of graded algebras \( \mathcal{I}=\oplus_{n}I_{n}T^{n} \) if \( 
I_{0}=\OO_{W} \) and \( I_{m}I_{n}\subset I_{m+n} \) for any \( 
m,n\in\mathbb{N} \).

We say that \( \mathcal{I}=\oplus_{n}I_{n}T^{n} \) is a Rees 
algebra over \( W \) if there is an affine open covering \( 
\{U_{i}\}_{i\in \Lambda} \) of \( W \) such that 
\begin{equation*}
    \mathcal{I}(U_{i})=\bigoplus_{n\in\mathbb{N}}I_{n}(U_{i})T^{n}
    \qquad \text{is a finitely generated } \OO_{W}(U_{i}) 
    \text{-algebra}
\end{equation*}
for any \( i\in\Lambda \).
Or equivalently, \( \mathcal{I}(U_{i}) \) is a Rees algebra in the 
ring \( \OO_{W}(U_{i})[T] \).
\end{parrafo}

Analogously we could define, \( \mathbb{N} \)-Rees algebras and \( \mathbb{Q}
\)-Rees algebras over \( W \).
\begin{definition} \label{DefQRAW}
A \( \mathbb{Q} \)-Rees algebra over \( W \) is denoted by
\begin{equation*}
    \mathcal{J}=\bigoplus_{a\in\QPlus} J_{a}T^{a}
\end{equation*}
where \( \{J_{a}\}_{a\in\QPlus} \) is a collection of coherent 
sheaves of ideals \( J_{a}\subset\OO_{W} \) such that:
\begin{enumerate}
    \item  \( J_{0}=\OO_{W} \),

    \item  \( J_{a}J_{b}\subset J_{a+b} \) for all \( 
    a,b\in\QPlus \),

    \item  \( J_{b}\subset J_{a} \) if \( a\leq b \) and

    \item  There exists an open covering of \( W \), by affine open sets 
    \( \{U_{i}\}_{i\in \Lambda} \) such that
    \( \mathcal{J}(U_{i}) \) is a \( \mathbb{Q} \)-Rees algebra in \( 
    \GrQ(\OO_{W}(U_{i})) \).
\end{enumerate}
\end{definition}

\begin{parrafo}
Let \( \mathcal{J} \) be a \( \mathbb{Q} \)-Rees algebra over \( W \) 
and consider an affine open set \( U\subset W \) such that \(
\mathcal{J}(U) \) is a \( \mathbb{Q} \)-Rees algebra in \(
\GrQ(\OO_{W}(U)) \).
For any open set \( V\subset U \) and for any point \( \xi\in U \)
there are natural maps:
\begin{equation*}
    \mathcal{J}(U)\to \mathcal{J}(V), \qquad 
    \mathcal{J}(U)\to \mathcal{J}_{\xi}
\end{equation*}
If \( f_{1}T^{a_{1}},\ldots,f_{r}T^{a_{r}}\in \mathcal{J}(U) \) are
generators of \( \mathcal{J}(U) \) as \( \mathbb{Q} \)-Rees algebra
(\ref{QReesAlg_fg}) then it is clear that the images of \(
f_{1}T^{a_{1}},\ldots,f_{r}T^{a_{r}} \) in \( \mathcal{J}(V) \) (resp.
\( \mathcal{J}_{\xi} \)) are also generators of the \( \mathbb{Q}
\)-Rees algebra \( \mathcal{J}(V) \) (resp. \( \mathcal{J}_{\xi} \)).
\end{parrafo} 

We say that \( \mathcal{J} \) is the zero algebra at a point \( \xi \)
if \( (J_{a})_{\xi}=0 \) for all \( a\in\QPos \). It 
follows from definition \ref{DefQRAW} that if \( 
(J_{b})_{\xi}=0 \) for some \( b\in\QPos \) then \( (J_{a})_{\xi}=0 
\) for all \( a\in\QPos \).

\begin{definition} \label{DefOrdRA}
Let \( \mathcal{J}=\oplus_{a}J_{a}T^{a} \) be a \( \mathbb{Q} \)-Rees 
algebra over \( W \) and let \( \xi\in W \) be a point.
Assume that \( \mathcal{J} \) is not the 
zero algebra at \( \xi \).
Define the order of \( \mathcal{J} \) at \( \xi \):
\begin{equation*}
    \ord(\mathcal{J})(\xi)=\inf_{a\in\QPos}\frac{\ord(J_{a})(\xi)}{a}
\end{equation*}
where \( \ord(J_{a})(\xi) \) denotes the order of the sheaf of ideals \( 
J_{a} \) at \( \xi \)
\begin{equation*}
    \ord(J_{a})(\xi)=\max\{n\in\mathbb{N}\mid (J_{a})_{\xi}\subset
    \mathfrak{m}_{\xi}^{n}\}
\end{equation*}
and \( \mathfrak{m}_{\xi}\subset\OO_{W,\xi} \) is the maximal ideal of the
local ring \( \OO_{W,\xi} \).

If \( \mathcal{J} \) is the zero algebra at \( \xi \) we may set \( 
\ord(\mathcal{J})(\xi)=\infty \).

We set the singular locus of \( \mathcal{J} \) as the set
\begin{equation*}
    \Sing(\mathcal{J})=\{\xi\in W\mid \ord(\mathcal{J})(\xi)\geq 1\}
\end{equation*}
\end{definition}

\begin{parrafo}
The order of \( \mathcal{J} \) at any point \( \xi \) is always a
rational number and it can be computed from a finite set of generators.
If \( f_{1}T^{a_{1}},\ldots,f_{r}T^{a_{r}}\in\mathcal{J}_{\xi} \) are 
generators of \( \mathcal{J} \) at \( \xi \) then
\begin{equation*}
    \ord(\mathcal{J})(\xi)=
    \min\left\{
    \frac{\ord(f_{1})}{a_{1}},\ldots,\frac{\ord(f_{r})}{a_{r}}
    \right\}
\end{equation*}
It follows that \( \ord(\mathcal{J}):W\to\QPlus \) is an 
upper-semicontinuous function. In particular the singular locus \( 
\Sing(\mathcal{J}) \) is a closed set.
\end{parrafo}

\begin{parrafo} \label{DefSimpleQRA}
The upper-semi-continuous function \( \ord(\mathcal{J}) \) will
stratify \( \Sing(\mathcal{J}) \) by locally closed sets and one may
focus to the maximum stratum, which is closed.

We say that a \( \mathbb{Q} \)-Rees algebra \( \mathcal{J} \) is
\emph{simple} if \( \ord(\mathcal{J})(\xi)=1 \) for any \(
\xi\in\Sing(\mathcal{J}) \).

We will see that, after a scaling operation (\ref{DefScaling}) we may 
reduce to the simple case.
\end{parrafo}

\section{Operations} \label{SeccOperations}

A very important concept with Rees algebras is integral closure.
Given a Rees algebra \( \mathcal{J}\subset\OO_{W}[T] \), the integral 
closure \( \mathcal{\bar{J}}\subset\OO_{W}[T] \) of \( \mathcal{J} \) 
in \( \OO_{W}[T] \) is the Rees algebra generated by all the elements of
\( \OO_{W}[T] \) that are integral over \( \mathcal{J} \).   
Note that it is well-known that the integral closure is a finitely 
generated \( \OO_{W} \)-algebra.

There is an open covering of \( W \), such that for any open set \( U \) 
of the covering, an element \( fT^{n}\in\OO_{W}(U)[T] \) belongs to \( 
\mathcal{\bar{J}} \) if and only if there exist \( m\in\mathbb{N} \) 
and a monic polynomial
\begin{equation*}
    p(Z)=Z^{m}+a_{1}Z^{m-1}+\cdots+a_{m}, \qquad
    \text{ with } a_{i}\in J_{in}, \quad i=1,\ldots,m
\end{equation*}
such that \( p(f)=0 \).

In fact we have that \( 
\mathcal{J} \) is equivalent to another Rees algebra \( 
\mathcal{J}_{1} \) if and only if \( 
\bar{\mathcal{J}}_{1}=\bar{\mathcal{J}} \).

We may define the analogous notion for \( \mathbb{Q} \)-Rees algebras.
\begin{definition}
Let \( \mathcal{J}=\oplus_{a}J_{a}T^{a} \) be a \( \mathbb{Q} \)-Rees 
algebra and let \( U \) be an open set of \( W \).
An element \( fT^{a}\in\OO_{W}(U)[T] \) is integral over \( \mathcal{J} \) 
if there exist \( m\in\mathbb{N} \) and a monic polynomial
\begin{equation*}
    p(Z)=Z^{m}+a_{1}Z^{m-1}+\cdots+a_{m}, \qquad
    \text{ with } a_{i}\in J_{ia}, \quad i=1,\ldots,m
\end{equation*}
such that \( p(f)=0\in\OO_{W}(U) \).

We set the integral closure \( \mathcal{\bar{J}} \) of \( \mathcal{J} \) as the \( 
\OO_{W} \)-algebra generated by all \( fT^{a} \) integral 
over \( \mathcal{J} \).
\end{definition}

From the definition we can not say that \( \mathcal{\bar{J}} \) is a 
\( \mathbb{Q} \)-Rees algebra, since the property of being finitely 
generated is not clear.
However, this fact will follow from radical closure. This concept was 
defined in \cite[2.1.3.1]{Kawanoue2007} and we will see the connection 
with \cite{LejeuneTeissier1974}.
% Now we will use a function order defined in \cite{LejeuneTeissier1974}.
\begin{definition} \cite{LejeuneTeissier1974}
Let \( \mathcal{J}=\oplus_{a}J_{a}T^{a} \) be a \( \mathbb{Q} \)-Rees 
algebra over \( W \). Fix a point \( \xi\in W \).
For an element \( f\in\OO_{W,\xi} \) we set
\begin{equation*}
    \nu_{\mathcal{J}}(f)=
    \sup\{a \mid fT^{a}\in\mathcal{J}\}
\end{equation*}
Note that \( \nu_{\mathcal{J}}(f) \) may be infinite if \( f=0 \).
\end{definition}

\begin{parrafo}
With the previous notation, for any \( f\in\OO_{W,\xi} \) consider 
the sequence
\begin{equation*}
    \left\{\frac{\nu_{\mathcal{J}}(f^{n})}{n}\right\}_{n=1}^{\infty}
\end{equation*}
Using the same arguments as in \cite[0.2.1]{LejeuneTeissier1974} we 
may prove that this sequence converges to some value in \( 
\mathbb{R}\cup\{\infty\} \).
So that it makes sense to set
\begin{equation*}
    \bar{\nu}_{\mathcal{J}}(f)=
    \lim_{n\to\infty}\frac{\nu_{\mathcal{J}}(f^{n})}{n}
\end{equation*}
\end{parrafo}

\begin{definition} \label{DefElemRadical}
Let \( \mathcal{J}=\oplus_{a}J_{a}T^{a} \) be a \( \mathbb{Q} \)-Rees 
algebra over \( W \).
An homogeneous element \( fT^{a}\in\OO_{W}[T] \) is radical over \( \mathcal{J} 
\) if \( \bar{\nu}_{\mathcal{J}}(f)\geq a \).

The radical saturation of \( \mathcal{J} \) is the \( \OO_{W} 
\)-algebra generated by all \( fT^{a} \) radical over \( 
\mathcal{J} \). We denote the radical saturation of \( \mathcal{J} \) 
as \( \mathfrak{R}(\mathcal{J}) \).
\end{definition}
From the definition, the radical saturation may not be a \( 
\mathbb{Q} \)-Rees algebra, since the condition of being finitely 
generated is not immediate.

\begin{parrafo}
It can be proved that \( fT^{a} \) is radical over \( \mathcal{J} \)
(\ref{DefElemRadical}) if and only if there are sequences \( 
\{a_{\ell}\}_{\ell=1}^{\infty} \) and \( 
\{n_{\ell}\}_{\ell=1}^{\infty} \) such that
\begin{itemize}
    \item  \( a_{\ell}\in\QPlus \), \( n_{\ell}\in\mathbb{N} \) for 
    all \( \ell \),

    \item  \( \lim_{\ell\to\infty}a_{\ell}=a\in\QPlus \), and

    \item  \( f^{n_{\ell}}T^{a_{\ell}n_{\ell}}\in\mathcal{J} \).
\end{itemize}
    
We recall the definition of radical saturation for idealistic 
filtrations \cite[2.1.3.1]{Kawanoue2007}. An idealistic filtration is 
said to be radical saturated if:
\begin{align*}
    \text{(radical)} &  & f^{m}\in J_{am},\ f\in R, 
    m\in\mathbb{N} & \Longrightarrow f\in J_{a}  \\
    \text{(continuity)} & & f\in J_{a_{\ell}},\ 
    \lim_{\ell\to\infty}a_{\ell}=a  & \Longrightarrow f\in J_{a}
\end{align*}
It follows that the two concepts: radical saturation of \( \mathbb{Q} 
\)-Rees algebras (\ref{DefElemRadical}) and radical saturation of 
idealistic filtrations \cite[2.1.3.1]{Kawanoue2007} coincide.
\end{parrafo}

\begin{proposition} \label{PropRadicalisQRA}
Let  \( \mathcal{J} \) be a  \( \mathbb{Q}\)-Rees algebra over  \( W \). 
The radical saturation \( \mathfrak{R}(\mathcal{J}) \) is a \( 
\mathbb{Q} \)-Rees algebra.
\end{proposition}

\begin{proof}
It follows from
\cite[2.3.2.4]{Kawanoue2007} where it is proved that radical saturation of a 
r.~f.~g. idealistic filtration is also a r.~f.~g. idealistic 
filtration.
In fact this result also appears in \cite[\S 4]{LejeuneTeissier1974}.
Both proofs use similar arguments inspired in \cite{Nagata1957}.

By the equivalence \ref{TresEquiv} we conclude then that \(
\mathfrak{R}(\mathcal{J}) \) is finitely generated and then it is a \(
\mathbb{Q} \)-Rees algebra.
\end{proof}

\begin{proposition} \label{IntegralisRadical}
The integral closure \( \mathcal{\bar{J}} \) is a \( \mathbb{Q} 
\)-Rees algebra. In fact \( 
\mathcal{\bar{J}}=\mathfrak{R}(\mathcal{J}) \).
\end{proposition}

\begin{proof}
It follows from \ref{PropRadicalisQRA} and \cite[2.3.2.7]{Kawanoue2007}.
See also \cite[\S 4]{LejeuneTeissier1974}.
\end{proof}

The order function is well defined up to equivalence of \( \mathbb{Q} 
\)-Rees algebras.
\begin{proposition}
Let \( \mathcal{J}_{1} \) and \( \mathcal{J}_{2} \) be two equivalent \( 
\mathbb{Q} \)-Rees algebras (\ref{DefIntEquivQRA}). Then for any \( 
\xi\in W \) we have \( 
\ord(\mathcal{J}_{1})(\xi)=\ord(\mathcal{J}_{2})(\xi) \).
\end{proposition}

\begin{proof}
It follows from the fact that
\( \ord(\mathcal{J})(\xi)=\ord(\mathfrak{R}(\mathcal{J}))(\xi) \), for 
any \( \xi\in W \), and \ref{IntegralisRadical}.
\end{proof}

Another important operation on \( \mathbb{Q} \)-Rees algebras is 
differential saturation. This notion appears also for Rees algebras 
(see \cite{Villamayor2008} and \cite{Villamayor2007}) and for 
idealistic filtrations \cite{Kawanoue2007}. See also 
\cite{Hironaka2003}.

Set \( \Diff_{W}^{m} \) to be the sheaf of differentials operators of 
order \( \leq m \).
We will say that a \( \mathbb{Q} \)-Rees algebra is differentially 
saturated if it is stable by the action of differentials, to be more 
precise:
\begin{definition} \label{DefDiffSat}
Let \( \mathcal{J}=\oplus_{a}J_{a}T^{a} \) be a \( \mathbb{Q} \)-Rees 
algebra. We say that \( \mathcal{J} \) is differentially saturated if
\begin{equation*}
    \Diff_{W}^{m}(J_{a})\subset J_{a-m}, \qquad \forall a\in\QPlus, \ 
    \forall m\in\mathbb{N}, \quad a\geq m
\end{equation*}
If \( \mathcal{J} \) is any \( \mathbb{Q} \)-Rees algebra, we denote 
\( \Diff(\mathcal{J}) \) to be the minimal \( \mathbb{Q} \)-Rees 
algebra, differentially saturated and containing \( \mathcal{J} \).
\end{definition}

\begin{parrafo}
The \( \mathbb{Q} \)-Rees algebra \( \Diff(\mathcal{J}) \) always exists and 
it may be computed from a set of generators of \( \mathcal{J} \).
If \( f_{1}T^{a_{1}},\ldots,f_{r}T^{a_{r}} \) is a set of generators 
of \( \mathcal{J}(U) \) (for some suitable open set \( U \)), then a 
set of generators of \( \Diff(\mathcal{J})(U) \) is
\begin{equation*}
    \{D(f_{i}) \mid i=1,\ldots,r,\ 
    D\in\Diff_{W}^{\lfloor a_{i}\rfloor}(U)\}
\end{equation*}
Note that, since \( \Diff_{W}^{m} \) is a locally free sheaf of 
finite rank, we could consider a finite set of generators for \( 
\Diff(\mathcal{J})(U) \).
\end{parrafo}

\begin{proposition}
Let \( \mathcal{J} \) be a \( \mathbb{Q} \)-Rees algebra. If \( 
\xi\in\Sing(\mathcal{J}) \) then \( 
\ord(\mathcal{J})(\xi)=\ord(\Diff(\mathcal{J}))(\xi) \).

If \( \xi\not\in\Sing(\mathcal{J}) \) then \( 
\ord(\Diff(\mathcal{J}))(\xi)=0 
\), or equivalently \( \Diff(\mathcal{J})_{\xi}=\GrQ(\OO_{W}) \).
\end{proposition}

\begin{proof}
See also \cite{Villamayor2008}.

First of all assume \( \xi\not\in\Sing(\mathcal{J}) \), then there exist
\( a\in\QPos \) and
\( f\in (J_{a})_{\xi} \) such that \( \ord(f)<a \).
We may assume that \( a\in\mathbb{N} \), since for any positive integer \( m
\), we have \( f^{m}\in(J_{ma})_{\xi} \) and \( \ord(f^{m})<ma \).
Then \( \ord(f)\leq a-1 \) and
there exists a differential operator \( D\in\Diff_{\OO_{W,\xi}}^{a-1} \)
such that \( D(f)\in\OO_{W,\xi} \) is a unit (or equivalently \(
\ord(D(f))=0 \)).
By definition~\ref{DefDiffSat} \( D(f)T\in\Diff(\mathcal{J})_{\xi} \) 
and this implies that \( \Diff(\mathcal{J})_{\xi}=\GrQ(\OO_{W}) \).

Now assume that \( \xi\in\Sing(\mathcal{J}) \). Set \(
\alpha=\ord(\mathcal{J})(\xi) \).
For any \( f\in\OO_{W,\xi} \) with \( fT^{a}\in\mathcal{J} \), any 
integer \( 0\leq k\leq a-1 \) and any differential operator \( D \) of
order \( k \), we have
\begin{equation*}
    \frac{\ord(D(f))}{a-k}\geq\frac{\ord(f)-k}{a-k}\geq \frac{\ord(f)}{a}\geq 1
\end{equation*}
and we conclude that \( \ord(\mathcal{J})(\xi)=\ord(\Diff(\mathcal{J}))(\xi) \).
Note that the hypothesis \( \xi\in\Sing(\mathcal{J}) \) is necessary 
in order to have \( \ord(f)\geq a \) in the chain of inequalities above.
\end{proof}

\begin{corollary} \label{CoroSingDiff}
\( \Sing(\mathcal{J})=\Sing(\Diff(\mathcal{J})) \).
\end{corollary}

\begin{definition}
Let \( \mathcal{J}_{1} \) and \( \mathcal{J}_{2} \) be two \( 
\mathbb{Q} \)-Rees algebras over \( W \).
We define
\( \mathcal{J}_{1}\odot\mathcal{J}_{2} \) as the \( \mathbb{Q} 
\)-Rees algebra generated by \( \mathcal{J}_{1}\cup\mathcal{J}_{2} \).
\end{definition}

\begin{parrafo}
If \( f_{1}T^{a_{1}},\ldots,f_{r}T^{a_{r}} \) are generators of \( 
\mathcal{J}_{1} \) and \( g_{1}T^{b_{1}},\ldots,g_{s}T^{b_{s}} \) are 
generators of \( \mathcal{J}_{2} \) it is clear that
\begin{equation*}
    f_{1}T^{a_{1}},\ldots,f_{r}T^{a_{r}}, 
    g_{1}T^{b_{1}},\ldots,g_{s}T^{b_{s}}
\end{equation*}
are generators of \( \mathcal{J}_{1}\odot\mathcal{J}_{2} \).

The notation \( \odot \) appeared in \cite{EncinasVillamayor2007}.
The use of \( \odot \) instead of \( \oplus \) was motivated since the
notation \( \mathcal{J}_{1}\oplus\mathcal{J}_{2} \) could be ambiguous.
\end{parrafo}

\begin{definition} \label{DefScaling}
Given \( \mathcal{J} \) and \( b\in\QPos \) we could do a scaling on 
the levels of \( \mathcal{J} \) as follows:
\begin{equation*}
    \mathcal{J}^{b}=\bigoplus_{a\in\QPlus}J_{a}T^{\frac{a}{b}}
\end{equation*}
\end{definition}

\begin{parrafo}
Note that \( fT^{a}\in\mathcal{J} \) iff \( fT^{\frac{a}{b}}\in\mathcal{J}^{b} \).
    
If \( f_{1}T^{a_{1}},\ldots,f_{r}T^{a_{r}} \) are generators of \( 
\mathcal{J} \) then \(
f_{1}T^{\frac{a_{1}}{b}},\ldots,f_{r}T^{\frac{a_{r}}{b}} \) are 
generators of \( \mathcal{J}^{b} \).

Note also that the order of \( \mathcal{J}^{b} \) is multiplied by \( 
b \)  by definition of the order function.
\begin{equation*}
    \ord(\mathcal{J}^{b})(\xi)=b\ord(\mathcal{J})(\xi), \qquad \forall
    \xi\in W.
\end{equation*}
\end{parrafo}

\begin{definition}
Let \( \mathcal{J} \) be a \( \mathbb{Q} \)-Rees algebra over \( W \)
and let \( \ell\in\QPlus \). Fix a (smooth) hypersurface \( H\subset W \).

We say that \( I(H)^{\ell} \) divides \( \mathcal{J} \) if 
\( J_{a}\subset I(H)^{\lceil a\ell\rceil} \) for any \( a\in\QPlus \).

Note that \( J_{a}=I(H)^{\lceil a\ell\rceil}I_{a} \) for some ideal \(I_{a} \). We set
 \( I(H)^{-\ell}\mathcal{J} \) to be the \( \mathbb{Q} \)-Rees 
algebra generated by \( \{I_{a}T^{a}\mid a\in\QPlus\} \).
\end{definition}

\begin{parrafo}
Note that in general \( \bigoplus_{a}I_{a}T^{a} \) is not a \( 
\mathbb{Q} \)-Rees algebra since condition (3) in \ref{DefQReesAlg} 
is not always satisfied.
\end{parrafo}

\begin{parrafo} \label{DivideHgen}
Let \( H\subset W \) be a (smooth) hypersurface, denote \( \nu_{H} \) 
the valuation associated to the hypersurface \( H \).

If \( f_{1}T^{a_{1}},\ldots,f_{r}T^{a_{r}} \) are generators of \( 
\mathcal{J} \) then \( I(H)^{\ell} \) divides \( \mathcal{J} \) if 
and only if \( \nu_{H}(f_{i})\geq a_{i}\ell \) for \( i=1,\ldots,r \).

Set \( \ell_{H} \) to be the supremum of all \( \ell\in\QPlus \) 
such that \( I(H)^{\ell} \) divides \( \mathcal{J} \).
It follows from the above characterization that
\begin{equation*}
    \ell_{H}=\min\left\{\frac{\nu_{H}(f_{i})}{a_{i}}\mid 
    i=1,\ldots,r\right\}
\end{equation*}
\end{parrafo}

\begin{parrafo} \label{DefTransformJ}
Let \( \mathcal{J} \) be a \( \mathbb{Q} \)-Rees algebra, and let \( 
C\subset\Sing(\mathcal{J}) \) be a smooth and closed set.
Let \( \Pi:W'\to W \) be the monoidal transformation with center \( C \).
The total transform of \( \mathcal{J} \) is \( 
\mathcal{J}^{\ast}=\oplus_{a}J_{a}^{\ast}T^{a} \), where \( 
J_{a}^{\ast} \) is the total transform of \( J_{a} \).
Note that \( \mathcal{J}^{\ast} \) is a \( \mathbb{Q} \)-Rees algebra.
Set \( H\subset W' \) to be the exceptional divisor of \( \Pi \).

Note also that \( I(H) \) divides \( \mathcal{J}^{\ast} \). The \( 
\mathbb{Q} \)-Rees algebra \( 
\mathcal{J}'=I(H)^{-1}\mathcal{J}^{\ast} \) is called the transform 
of \( \mathcal{J} \).

Let \( f_{1}T^{a_{1}},\ldots,f_{r}T^{a_{r}} \) be generators of \(
\mathcal{J} \). Denote \( f_{i}^{\ast} \) the total transform of \(
f_{i} \), via the morphism \( \OO_{W}\to\OO_{W'} \).
From the fact \( C\subset\Sing(\mathcal{J}) \) it follows that
\( f_{i}^{\ast}=x^{a_{i}}g_{i} \) for some \( g_{i}\in\OO_{W'} \) and 
where \( I(H)=(x) \).
The transform \( \mathcal{J}' \) is generated by \(
g_{1}T^{a_{1}},\ldots,g_{r}T^{r} \).
\end{parrafo}

\begin{parrafo} \label{DefNonMonomPart}
Let \( E=\{H_{1},\ldots,H_{N}\} \) be a set of smooth hypersurfaces in \( W 
\) having only normal crossings and let \(
\mathcal{J}=\oplus_{a}J_{a}T^{a} \) be a \( \mathbb{Q} \)-Rees algebra
over \( W \).

For \( i=1,\ldots,r \), set \( \ell_{H_{i}} \) to be the maximum such that
\( I(H_{i})^{\ell_{H_{i}}} \) divides \( \mathcal{J} \) 
(\ref{DivideHgen}).
Note that for any \( a\in\QPlus \)
\begin{equation*}
    J_{a}=I(H_{1})^{\lceil a\ell_{H_{1}} \rceil} \cdots
    I(H_{N})^{\lceil a\ell_{H_{N}} \rceil} I_{a}
\end{equation*}
where \( I_{a}\subset\OO_{W} \) is a sheaf of ideals.

Set \( \mathcal{I}=E^{-1}\mathcal{J} \) to be the \( \mathbb{Q}
\)-Rees algebra generated by \( \{I_{a}T^{a}\mid a\in\QPlus\} \).
We will call \( E^{-1}\mathcal{J} \) the non monomial part of \(
\mathcal{J} \) with respect to \( E \).
\end{parrafo}

\begin{parrafo}
Consider some point or some affine open set of \( W \).
Suppose that \( f_{1}T^{a_{1}},\ldots,f_{r}T^{a_{r}} \) are generators of \(
\mathcal{J} \) and \( I(H_{j})=(x_{j}) \), \( j=1,\ldots,N \).
Set \( c_{i,j}=\nu_{H_{j}}(f_{i}) \), for \( i=1,\ldots,r \) and \(
j=1,\ldots,N \). We have
\begin{equation*}
    f_{i}=x_{1}^{c_{i,1}}\cdots x_{N}^{c_{i,N}}g_{i}
\end{equation*}
where the equation \( g_{i}\not\in (x_{j}) \) for any \( j=1,\ldots,N \).
Set \( c'_{i,j}=c_{i,j}-\lceil a_{i}\ell_{H_{j}}\rceil \).
Note that \( c'_{i,j}\geq 0 \) since
\( \ell_{H_{j}}=\min\{\frac{c_{i,j}}{a_{i}}\mid i=1,\ldots,r\} \).
The \( \mathbb{Q} \)-Rees algebra \( E^{-1}\mathcal{J} \) is generated
by 
\begin{equation*}  
h_{1}T^{a_{1}},\ldots,h_{r}T^{a_{r}}
\end{equation*}
where
\begin{equation*}
    h_{i}=x_{1}^{c'_{i,1}}\cdots x_{N}^{c'_{i,N}}g_{i}, \qquad i=1,\ldots,r.
\end{equation*}
\end{parrafo}

\section{Coefficient algebras.} \label{SeccCoeff}

For any morphism \( V\to W \) of smooth varieties we have the natural morphism of sheaves \( 
\OO_{W}\to\OO_{V} \). We may also consider a morphism \(
\GrQ(\OO_{W})\to\GrQ(\OO_{V}) \).

\begin{definition}
Fix \( W \) a smooth variety of pure dimension and a \( \mathbb{Q}
\)-Rees algebra \( \mathcal{J} \) over \( W \).
If \( V\subset W \) is a smooth subvariety of pure dimension,
we set the coefficient algebra of \( \mathcal{J} \) with respect to \( 
V \) as follows
\begin{equation*}
    \Coeff_{V}(\mathcal{J})=\Diff(\mathcal{J})\GrQ(\OO_{V})
\end{equation*}
\end{definition}

\begin{parrafo}
It is well known that the coefficient algebra describes the same
singular locus as \( \mathcal{J} \). If \( V\subset W \) is a smooth
subvariety of \( W \) then
\begin{equation} \label{EqCoeff}
    \Sing(\mathcal{J})\cap V=\Sing(\Coeff_{V}(\mathcal{J}))
\end{equation}

Moreover, equality \ref{EqCoeff} is preserved by 
transformations.
To be more precise, set \( \mathcal{A}=\Coeff_{V}(\mathcal{J}) \) and consider a
sequence of transformations (\ref{DefTransformJ}):
\begin{equation} \label{SeqTransIntersec}
    \begin{array}{rcccccccc}
        W=W_{0} & \longleftarrow & W_{1} & \longleftarrow & W_{2} & \longleftarrow & 
	\cdots & \longleftarrow & W_{N}  \\
         &  & H_{1} &  & H_{2} &  &  &  & H_{N}  \\
        \mathcal{J}=\mathcal{J}_{0} &  & \mathcal{J}_{1} &  & \mathcal{J}_{2} &  &  &  &
	\mathcal{J}_{N} \\
	V=V_{0} & \longleftarrow & V_{1} & \longleftarrow & V_{2} & \longleftarrow & 
	\cdots & \longleftarrow & V_{N} \\
	\mathcal{A}=\mathcal{A}_{0} &  & \mathcal{A}_{1} &  &
	\mathcal{A}_{2} &  &  &  & \mathcal{A}_{N}
    \end{array}
\end{equation}
where for any \( i=1,\ldots,N \),
\begin{itemize}
    \item \( \Pi_{i}:W_{i}\to W_{i-1} \) is a monoidal transformation
    with center \( C_{i-1}\subset \Sing(\mathcal{J}_{i-1})\cap
    V_{i-1}\subset W_{i-1} \),

    \item  \( H_{i}\subset W_{i} \) is the exceptional divisor of \(
    \Pi_{i} \),

    \item  \( \mathcal{J}_{i}=I(H_{i})^{-1}\mathcal{J}_{i-1}^{*} \) is the
    transform of \( \mathcal{J}_{i-1} \),

    \item  \( V_{i} \) is the strict transform of \( V_{i-1} \) and
    
    \item \( A_{i}=I(H_{i}\cap V_{i})^{-1}\mathcal{A}_{i-1}^{*}  \) is the 
    transform of \( \mathcal{A}_{i-1} \).
\end{itemize}
Then we have the equality
\begin{equation*}
    \Sing(\mathcal{J}_{N})\cap V_{N}=\Sing(\mathcal{A}_{N})
\end{equation*}
in fact inductively we have
\begin{equation*}
    \Sing(\mathcal{J}_{i})\cap V_{i}=\Sing(\mathcal{A}_{i}), \qquad
    i=0,\ldots,N
\end{equation*}
This fact follows from \ref{LemaGiraud}.
\end{parrafo}

\begin{theorem} \label{LemaGiraud}
(Giraud)
Let \( \mathcal{J} \) be a \( \mathbb{Q} \)-Rees algebra over \( W \) 
and \( W'\to W \) be a monoidal transformation with center \(
C\subset\Sing(\mathcal{J}) \).

Set \( \mathcal{D}=\Diff(\mathcal{J}) \) the differential saturation
of \( \mathcal{J} \).
Recall that \( \Sing(\mathcal{J})=\Sing(\mathcal{D}) \) by \ref{CoroSingDiff}.
Denote by \( \mathcal{J}'=I(H)^{-1}\mathcal{J}^{*}  \)
and \( \mathcal{D}'=I(H)^{-1}\mathcal{D}^{*}  \) the transforms of \(
\mathcal{J} \) and \( \mathcal{D} \), respectively.
Then 
\begin{equation*}
    \mathcal{J}'\subset \mathcal{D}'\subset \Diff(\mathcal{J}')
\end{equation*}
\end{theorem}

\begin{proof}
    See \cite{EncinasVillamayor2007}.
\end{proof}

A key fact for proving resolution of singularities in characteristic
zero is the existence (locally) of hypersurfaces \( V\subset W \) such
that the singular locus of a simple \( \mathbb{Q} \)-Rees algebra \(
\mathcal{J} \) (\ref{DefSimpleQRA}) is included in \( V \).
\begin{theorem} \label{ThContMaxLocal}
Assume that characteristic of the ground field \( k \) is zero.

Let \( \mathcal{J} \) be a simple \( \mathbb{Q} \)-Rees algebra
(\ref{DefSimpleQRA}) over \( W \) and \( \xi\in\Sing(\mathcal{J}) \).
There is an open set \( \xi\in U\subset W \) and a smooth hypersurface \( 
\xi\in V\subset U \) such that
\begin{equation} \label{EqContMaxProp}
    I_{U}(V)T\subset \Diff(\mathcal{J})|_{U}
\end{equation}
where \( I_{U}(V) \) denotes the ideal sheaf defined by \( V \) in \( 
U \).
\end{theorem}

\begin{proof}
See \cite{EncinasVillamayor2007}.
The algebra \( \mathcal{J} \) is simple, so that there is an equation \( 
fT^{b}\in\mathcal{J} \) with \( \ord(f)=b\in\mathbb{Z} \).
Since we are in characteristic zero, there is a differential operator 
\( D \) of order \( b-1 \) such that \( D(f) \) has order one.
Note that \( D(f)T\in\Diff(\mathcal{J}) \).
At a suitable neighborhood the equation \( D(f) \) defines a smooth 
hypersurface \( V \).
\end{proof}

Theorem \ref{ThContMaxLocal} does not hold in positive
characteristic case and this is one of the main obstructions to find a
proof in the general case.

\begin{parrafo} \label{ContMaxEstable}
Let \( \mathcal{J} \) be a simple \( \mathbb{Q} \)-Rees algebra over \( W \).
Assume that there exists a smooth hypersurface \( V\subset W \) such that
\( I(V)T\subset \Diff(\mathcal{J}) \).
It follows from \ref{ThContMaxLocal} that this assumption may be always satisfied
in all open sets of a suitable open covering of \( W \).

Note that \( I(V)T\subset \Diff(\mathcal{J}) \) implies that \(
\Sing(\mathcal{J})\subset V \).
Set \( \mathcal{A}=\Coeff_{V}(\mathcal{J}) \), by \ref{EqCoeff} we have \(
\Sing(\mathcal{J})=\Sing(\mathcal{A}) \). In fact for a sequence of
transformations as in \ref{SeqTransIntersec} we have for any \(
i=1,\ldots,N \):
\begin{itemize}
    \item  \( I(V_{i})T\subset\Diff(\mathcal{J}_{i}) \),
    
    \item \( \Sing(\mathcal{J}_{i})\subset V_{i} \) and

    \item  \( \Sing(\mathcal{J}_{i})=\Sing(\mathcal{A}_{i}) \).
\end{itemize}
\end{parrafo}

Theorem \ref{ThContMaxLocal} allows to choose hypersurfaces \( V \)
with the property \ref{EqContMaxProp}. This is the inductive argument in 
the proof of resolution of singularities. We replace the simple algebra \(
\mathcal{J} \) in \( W \) with the algebra \( \Coeff_{V}(\mathcal{J}) \)
in \( V \).
But we have to prove that the total procedure will be independent of
the choice of \( V \).
This problem was originally solved by Hironaka, by considering an
equivalence relation of objects and proving that the procedure of
resolution depends only on the equivalence class of the algebra \(
\mathcal{J} \).

An alternative path is the solution given by W{\l}odarczyk
\cite{Wlodarczyk2005}.
\begin{theorem} \label{ThJarek}
Let \( \mathcal{J} \) be a simple \( \mathbb{Q} \)-Rees algebra. Fix a
point \( \xi\in\Sing(\mathcal{J}) \).
Assume that \( V_{1} \) and \( V_{2} \) are two smooth hypersurfaces
of \( W \), \( \xi\in V_{1}\cap V_{2} \).
Then there is an \'etale neighborhood \( U \) of \( \xi \) in \( W \) 
and an automorphism \( \varphi:U\to U \) such that
\begin{equation*}
    \varphi(\Diff(\mathcal{J})|_{U})=\Diff(\mathcal{J})|_{U},
    \qquad
    \varphi(V_{1})=V_{2}
\end{equation*}
\end{theorem}

\begin{proof}
We repeat the proof of \cite{Wlodarczyk2005} in terms of \( 
\mathbb{Q} \)-Rees algebras.

Set \( I(V_{i})=(u_{i}) \) for some \( u_{i}\in\OO_{W,\xi} \), \(
i=1,2 \). There are \( x_{2},\ldots,x_{d}\in\OO_{W,\xi} \) such that \( 
u_{1},x_{2},\ldots,x_{d} \) and \( u_{2},x_{2},\ldots,x_{d} \) are
both regular systems of parameters at \( \OO_{W,\xi} \).

Consider the automorphism \( \varphi^{\sharp} \) of \(
\hat{\OO}_{W,\xi} \) sending \( u_{1} \) to \( u_{2} \) and fixing \(
x_{2},\ldots,x_{d} \).  Note that this automorphism can be lifted to a suitable
\'etale neighborhood. But, for simplification, we will consider all ideals 
in the completion \( \hat{\OO}_{W,\xi} \).
Denote \( \hat{\mathcal{J}}=\mathcal{J}\GrQ(\hat{\OO}_{W,\xi}) \).

Let \( fT^{a}\in\Diff(\mathcal{J})_{\xi} \).  Consider the image of
\( f \) in the completion \( f\in\hat{\OO}_{W,\xi} \), note that \( f
\) is a power series \( f=F(u_{1},x_{2},\ldots,x_{d}) \), then \(
\varphi^{\sharp}(F(u_{1},x_{2},\ldots,x_{d}))=
F(u_{2},x_{2},\ldots,x_{d}) \).
By assumption \(
u_{i}T\in\Diff(\hat{\mathcal{J}}) \), \( i=1,2
\), set \( h=u_{2}-u_{1} \).  We have \(
F(u_{2},x_{2},\ldots,x_{d})=F(u_{1}+h,x_{2},\ldots,x_{d}) \) and
\begin{equation*}
    F(u_{1}+h,x_{2},\ldots,x_{d})=
    \sum_{i=0}^{\infty}
    h^{i}\frac{1}{i!}\frac{\partial^{i}F}{\partial u_{1}^{i}}
    (u_{1},x_{2},\ldots,x_{d})
\end{equation*}
Now note that \( h^{i}T^{i}\in\Diff(\hat{\mathcal{J}})
\) for all \( i\geq 0 \), and \(
FT^{a}\in\Diff(\hat{\mathcal{J}}) \) implies that \(
\dfrac{\partial^{i}F}{\partial u_{1}^{i}}T^{a-i}\in
\Diff(\hat{\mathcal{J}}) \) for \( 0\leq i<a \).  We
conclude \(
\varphi^{\sharp}(f)T^{a}\in\Diff(\hat{\mathcal{J}}) \)
and then \( \varphi^{\sharp}(f)T^{a}\in \Diff(\mathcal{J})_{\xi} \)
at a suitable \'etale neighborhood.
\end{proof}

\section{Elimination algebras.} \label{SeccElimin}

Villamayor has introduced in the paper \cite{Villamayor2007} the 
concept of elimination algebra.
The coefficient algebra is defined for an inclusion \( V\subset W \)
and the elimination algebra will be defined for a smooth morphism \(
W\to V \).
In characteristic zero both algebras will encode the same information.

\begin{definition} \label{DefElimAlg}
Let \( \mathcal{J} \) be a \( \mathbb{Q} \)-Rees algebra over \( W \),
with pure dimension \( d=\dim(W) \).
Let \( V \) a regular algebraic variety of pure dimension \( d-1 \) and \( 
\beta:W\to V \) be a smooth morphism.

Consider the natural sheaf homomorphism \( \OO_{V}\to\OO_{W} \), which
induces an homomorphism \( \GrQ(\OO_{V})\to\GrQ(\OO_{W}) \).
The elimination algebra of \( \mathcal{J} \) in \( V \) is
\begin{equation*}
    \mathcal{R}_{V}(\mathcal{J})=\mathcal{J}\cap\GrQ(\OO_{V})
\end{equation*}
\end{definition}

The elimination algebra has good properties when the algebra \(
\mathcal{J} \) is simple and differentially saturated.
In fact in \cite{Villamayor2007} the elimination algebra is only
defined for simple algebras saturated by the relative differentials with
respect to the morphism \( \beta \). Note that in \cite{Villamayor2007}
the elimination algebra is constructed using a universal algebra in
terms of symmetric polynomial. This way properties related to integral 
closure can be proved.

\begin{theorem} \label{ThBetaSing}
\cite{Villamayor2007}
Let \( \mathcal{J} \) be a simple \( \mathbb{Q} \)-Rees algebra over \( 
W \), and let \( \beta:W\to V \) be a smooth morphism, \(
\dim{V}=\dim{W}-1 \).
Assume that \( \mathcal{J}=\Diff(\mathcal{J}) \) then
\begin{equation*}
    \Sing(\mathcal{R}_{V}(\mathcal{J}))=\beta(\Sing(\mathcal{J}))
\end{equation*}
Moreover, \( \beta \) is 1-1 between the points of 
\( \Sing(\mathcal{R}_{V}(\mathcal{J})) \) and 
\( \beta(\Sing(\mathcal{J})) \).
\end{theorem}

Unfortunately the equality in \ref{ThBetaSing} does not hold, in 
general, after monoidal transformation.
\begin{example} \label{EjemElimina}
Consider a field \( k \) of characteristic two.
Set \( W=\Spec(k[x,y,z]) \), \( V=\Spec(k[y,z]) \) and \( \beta:W\to V \) 
the usual projection. Consider the \( \mathbb{Q} \)-Rees 
algebra \( \mathcal{J} \) generated by \( (x^{2}+y^{2}z)T^{2} \).
The differential saturation is generated by
\begin{equation*}
    (x^{2}+y^{2}z)T^{2},\ y^{2}T
\end{equation*}
The elimination algebra of the differential saturation
\( \mathcal{A}=\mathcal{R}_{V}(\Diff(\mathcal{J})) \) is 
generated by \( y^{2}T \).

Now consider the blowing-up with center at the origin of \( W \) and 
consider the affine chart where the ideal of the exceptional divisor 
is \( z \).
The transform \( \mathcal{J}_{1} \) of \( \mathcal{J} \) is generated by \( 
(x^{2}+y^{2}z)T^{2} \).
The transform \( \mathcal{A}_{1} \) of the elimination algebra \( 
\mathcal{A} \) is generated by \( y^{2}zT \).
Note that \( \Sing(\mathcal{J}_{1}) \) is a line and \( 
\Sing(\mathcal{A}_{1}) \) is a union of two lines with
\begin{equation*}
    \Sing(\mathcal{J}_{1})\subset \Sing(\mathcal{A}_{1})
\end{equation*}
\end{example}

If the characteristic of \( k \) is zero then equality in
theorem~\ref{ThBetaSing} is stable by monoidal transformation.  In
fact, after an \'etale extension, the coefficient and elimination
algebra are isomorphic.
\begin{theorem} \label{CoeffIsoElim}
Assume that the characteristic of the ground field \( k \) is zero.

Let \( \mathcal{J} \) be a simple \( \mathbb{Q} \)-Rees algebra.
Fix a closed point \( \xi\in\Sing(\mathcal{J}) \).
There is an \'etale neighborhood of \( \xi \), \( U \) in \( W \),
a smooth hypersurface \( V\subset U \) and a retraction \( \beta:U\to 
V \) such that
\begin{equation*}
    \Coeff_{V}(\mathcal{J}|_{U})=\mathcal{R}_{V}(\Diff(\mathcal{J}|_{U}))
\end{equation*}
\end{theorem}

\begin{proof}
By theorem~\ref{ThContMaxLocal} there is an equation \( 
z\in\OO_{W,\xi} \) with \( zT\in\Diff(\mathcal{J}) \).
Consider \( x_{2},\ldots,x_{d}\in\OO_{W,\xi} \) such that \( 
z,x_{2},\ldots,x_{d} \) is a regular system of parameters.
Set \( \hat{\mathcal{J}}=\mathcal{J}\GrQ(\hat{\OO}_{W,\xi}) \).
We will prove that
\begin{equation*}
    \Coeff_{\hat{V}}(\hat{\mathcal{J}})=
    \mathcal{R}_{\hat{V}}(\Diff(\hat{\mathcal{J}}))
\end{equation*}
where \( V \) is the hypersurface defined by \( z=0 \).
Note that we are considering \( \hat{\OO}_{W,\xi} \) as the power 
series ring \( k'[[z,x_{2},\ldots,x_{d}]] \) and \( 
\OO_{\hat{V},\xi} \) as the power series ring \( 
k'[[x_{2},\ldots,x_{d}]] \), where \( k'\supset k \) is the residual 
field at \( \xi \). The retraction to \( \hat{V} \) corresponds to 
the inclusion of those power series rings.

The inclusion 
\( \mathcal{R}_{\hat{V}}(\Diff(\hat{\mathcal{J}})) \subset
\Coeff_{\hat{V}}(\hat{\mathcal{J}}) \) is now clear.

Assume that \( fT^{b}\in\hat{\mathcal{J}} \). We 
can express \( f \) as
\begin{equation*}
    f=a_{0}(x)+a_{1}(x)z+a_{2}(x)z^{2}+\cdots+a_{b-1}(x)z^{b-1}+a_{b}(x,z)z^{b}
\end{equation*}
where \( a_{i}(x)\in k'[[x_{2},\ldots,x_{d}]] \), for \( 
i=0,\ldots,b-1 \) and \( a_{b}(x,z)\in k'[[z,x_{2},\ldots,x_{d}]] \).
Note that \( a_{0}(x)T^{b}\in \Coeff_{\hat{V}}(\hat{\mathcal{J}}) \).

It is enough to prove that \( a_{0}(x)T^{b}\in \Diff(\hat{\mathcal{J}}) \).
In fact we will see by descending induction that \( 
a_{i}(x)T^{b-i}\in \Diff(\hat{\mathcal{J}}) \), for \( i=0,\ldots,b-1 \).

Recall that \( zT\in\Diff(\hat{\mathcal{J}}) \) and then
\begin{equation*}
    \dern{f}{z}{b-1}=(b-1)!(a_{b-1}(x)+b a_{b}(x,z)z)
    \qquad \Longrightarrow \qquad
    a_{b-1}(x)T\in \Diff(\hat{\mathcal{J}})
\end{equation*}
Fix \( i<b-1 \), assume that \( 
a_{j}(x)T^{b-j}\in\Diff(\hat{\mathcal{J}}) \) for \( j=i+1,\ldots,b-1 \).
It follows from the expression
\begin{equation*}
    \dern{f}{z}{i}=i!\left(a_{i}(x)+
    \sum_{j=i+1}^{b}\binom{j}{i}a_{j}(x)z^{j-i}
    \right)
\end{equation*}
that \( a_{i}(x)T^{b-i}\in\Diff(\hat{\mathcal{J}}) \), as required.
\end{proof}

\section{Algorithm of resolution in characteristic zero case.}
\label{SeccAlgorithm}

Along this section we assume that the characteristic of the field \( 
k \) is zero.
\medskip

We will describe here an algorithm of resolution for \( \mathbb{Q} 
\)-Rees algebras. It is inspired in \cite{EncinasHauser2002}.

\begin{definition} \label{DefResolQRA}
Let \( \mathcal{J} \) be a \( \mathbb{Q} \)-Rees algebra over \( W 
\). Assume that \( E \) is a set of hypersurfaces of \( W \) having 
only normal crossings.

A resolution of \( \mathcal{J} \) over \( (W,E) \) is a sequence of 
monoidal transformations:
\begin{equation} \label{SeqResolQRA}
    (W,E)=(W_{0},E_{0})\longleftarrow (W_{1},E_{1})\longleftarrow 
    \cdots \longleftarrow (W_{N},E_{N})
\end{equation}
such that, for \( i=0,\ldots,N-1 \)
\begin{itemize}
    \item  If \( \mathcal{J}_{i} \) is the transform of \( 
    \mathcal{J} \) in \( W_{i} \) then \( W_{i+1}\to W_{i} \) is the 
    monoidal transformation with center \( C_{i}\subset 
    \Sing(\mathcal{J}_{i}) \),
    
    \item \( C_{i} \) has normal crossings with \( E_{i} \),

    \item  \( E_{i+1} \) consists of all strict transforms of \( 
    E_{i} \) and the exceptional divisor of \( W_{i+1}\to W_{i} \) and
    
    \item \( \Sing(\mathcal{J}_{N})=\emptyset \).
\end{itemize}
\end{definition}
A log-resolution of an ideal \( J\subset\OO_{W} \) can be achieved by 
a resolution of the \( \mathbb{Q} \)-Rees algebra generated by \( JT \).

We will construct the sequence \ref{SeqResolQRA} inductively on the
number of blowing-ups and the dimension \( d \) of \( W \).

\begin{parrafo}\label{PropSeqFc}
At every step we will define an upper-semi-continuous function
\(
\Fc_{i}^{(d)}:W_{i}\to \Lambda \) where \(
\Lambda=(\QPlus\times\mathbb{N})^{\mathbb{N}} \) is ordered lexicographically.
The function \( \Fc_{i}^{(d)} \) will depend on the previous steps, say
the functions \( \Fc_{0}^{(d)},\ldots,\Fc_{i-1}^{(d)} \).

In fact the situation is local.
If \( \xi_{i}\in W_{i} \), the definition of the value 
\( \Fc_{i}(\xi_{i}) \) is local on the sequence \ref{SeqResolQRA}.
It depends only on the values of \(
\Fc_{0}(\xi_{0}),\ldots,\Fc_{i-1}(\xi_{i-1}) \), where \(
\xi_{0},\ldots,\xi_{i-1} \) are the images of \( \xi_{i} \) at \(
W_{0},\ldots,W_{i-1} \), respectively, and on the stalks \( 
\mathcal{J}_{0,\xi_{0}},\ldots,\mathcal{J}_{i-1,\xi_{i-1}} \).

Each center \( C_{i}\subset W_{i} \) will be the set of points where
the function \( \Fc_{i}^{(d)} \) is maximum:
\begin{equation*}
    C_{i}=\MaxB{\Fc_{i}^{(d)}}=\{\xi\in W_{i}\mid
    \Fc_{i}^{(d)}(\xi)=\max\Fc_{i}^{(d)}\}
\end{equation*}
Moreover, it can be proved that sequence \ref{SeqResolQRA} together with 
functions \( \Fc \) have the following properties:
\begin{enumerate}
    \item  \( C_{i}=\MaxB{\Fc_{i}^{(d)}} \),

    \item  \( \max\Fc_{0}^{(d)}>\max\Fc_{1}^{(d)}>\cdots 
    \max\Fc_{N-1}^{(d)} \) and

    \item  if \( \xi_{i}\in W_{i}\setminus C_{i} \) then \( \xi_{i} \) 
    identifies with a point \( \xi_{i+1}\in W_{i+1} \) and
    \( \Fc_{i}^{(d)}(\xi_{i})=\Fc_{i+1}^{(d)}(\xi_{i+1}) \).
\end{enumerate}
\end{parrafo}

\begin{parrafo} \label{FuntorProp}
We we will require also a stability property with smooth morphisms.

Let \( \mathcal{J} \) be a \( \mathbb{Q} \)-Rees algebra over \( 
(W,E) \), let \( W'\to W \) be a smooth morphism. Set \( E' \) to be the 
set of hypersurfaces of \( W' \) obtained by
the pullback of \( E \).
Set \( \mathcal{J}' \) to be the \( \mathbb{Q} \)-Rees algebra 
obtained also by pullback.

In this situation we have two sequences, say the resolutions of \( 
\mathcal{J} \) and \( \mathcal{J}' \).

Denote the resolution of \( \mathcal{J} \):
\begin{equation}
    \begin{array}{ccccccc}
        (W_{0},E_{0}) & \longleftarrow & (W_{1},E_{1}) & 
	\longleftarrow & \cdots & \longleftarrow & (W_{N},E_{N})  \\
        \mathcal{J}_{0} &  & \mathcal{J}_{1} &  &  &  & \mathcal{J}_{N}  \\
        \Fc_{0}^{(d)}(\mathcal{J}) &  & \Fc_{1}^{(d)}(\mathcal{J}) &  
	&  &  & 
    \end{array}
    \label{EqFuncResol1}
\end{equation}
where \( \Fc_{i}^{(d)}(\mathcal{J}):W_{i}\to\Lambda \), \( 
i=0,\ldots,N-1 \), are the functions associated to the resolution of 
\( \mathcal{J} \).

On the other hand denote the resolution of \( \mathcal{J}' \):
\begin{equation}
    \begin{array}{ccccccc}
        (W'_{0},E'_{0}) & \longleftarrow & (W'_{1},E'_{1}) & 
	\longleftarrow & \cdots & \longleftarrow & (W'_{N},E'_{N'})  \\
        \mathcal{J}'_{0} &  & \mathcal{J}'_{1} &  &  &  & 
	\mathcal{J}'_{N'}  \\
        \Fc_{0}^{(d')}(\mathcal{J}') &  & \Fc_{1}^{(d')}(\mathcal{J}') &  
	&  &  & 
    \end{array}
    \label{EqFuncResol2}
\end{equation}
where \( \Fc_{i}^{(d')}(\mathcal{J}'):W'_{i}\to\Lambda \), \( 
i=0,\ldots,N'-1 \), are the functions associated to the resolution of 
\( \mathcal{J'} \).

The stability property says that the sequence \ref{EqFuncResol2} is 
the pullback, via \( W'\to W \), of the sequence \ref{EqFuncResol1}; 
and functions take the same values
\begin{equation*}
    \Fc_{i}^{(d')}(\mathcal{J}')(\xi')= 
    \Fc_{i}^{(d)}(\mathcal{J})(\xi)
\end{equation*}
for any \( \xi'\in W'_{i} \) and any \( i=0,1,\ldots,N'-1 \).
Where \( \xi'\in W'_{i} \) maps to \( \xi\in W_{i} \).
\end{parrafo}

\begin{parrafo}
\textbf{Dimension one case.}
If \( \dim{W}=1 \), the singular locus \( \mathcal{J} \) is a finite
number of points in \( W \).
We set \( \Fc_{0}^{(1)}=\ord(\mathcal{J}) \).
The blowing-up with center \( C_{0}=\MaxB{\Fc_{0}^{(1)}} \), \( W_{1}\to 
W_{0}=W \), is an isomorphism, but the transform \( \mathcal{J}_{1} \)
is a different \( \mathbb{Q} \)-Rees algebra.
Note that \( \mathcal{J}_{1}=I(C_{0})^{-1}\mathcal{J}^{*} \)
and if \( \xi\in C_{0} \) then \( 
\ord(\mathcal{J}_{1})(\xi)=\ord(\mathcal{J}_{0})(\xi)-1 \).

If \( \Sing(\mathcal{J}_{1})\neq\emptyset \) then \( 
\max\Fc_{1}^{(1)}\geq 1 \). Set \( 
\Fc_{1}^{(1)}=\ord(\mathcal{J}_{1}) \) and continue with this 
procedure.
It is easy to prove that we obtain a sequence as in \ref{SeqResolQRA} 
with the required properties in \ref{PropSeqFc}.
\end{parrafo}

\begin{parrafo} \label{HipotInducdim}
In what follows we will fix a dimension \( d>1 \) and assume that we 
have constructed resolution of \( \mathbb{Q} \)-Rees algebras over 
varieties of dimension \( d-1 \). The constructed procedure satisfies 
properties in \ref{DefResolQRA} and \ref{PropSeqFc}, and
also stability property \ref{FuntorProp} for smooth morphisms of 
relative dimension zero.
\end{parrafo}

\begin{parrafo} \label{Step0}
\textbf{Initial step.}
Fix a point \( \xi_{0}\in\Sing(\mathcal{J}_{0}) \).
We construct the value \( \Fc_{0}^{(d)}(\xi_{0}) \) as follows:

Set \( \omega=\omega_{\xi_{0},0}^{(d)}=\ord(\mathcal{J}_{0})(\xi_{0}) \) and set
\begin{equation*}
    \mathcal{I}_{\xi_{0},0}=\mathcal{J}_{0},\qquad
    \mathcal{P}_{\xi_{0},0}= 
    \mathcal{I}_{\xi_{0},0}^{\frac{1}{\omega}}, 
    \qquad \mathcal{T}_{\xi_{0},0}=\mathcal{P}_{\xi_{0},0}\odot 
    \mathcal{E}_{\xi_{0},0}
\end{equation*}
where \( \mathcal{E}_{\xi_{0},0} \) is the \( \mathbb{Q} \)-Rees 
algebra generated by \( \{I(H)T\mid H\in E_{0},\ \xi_{0}\in H\} \).

The \( \mathbb{Q} \)-Rees algebra \( \mathcal{P}_{\xi_{0},0} \) has
order one at \( \xi_{0} \).  We may consider an open neighborhood of \(
\xi_{0} \) such that the algebra  \( \mathcal{P}_{\xi_{0},0} \) is simple.  By \ref{ThContMaxLocal}
we could consider a more suitable neighborhood \( U \) in order to choose a 
smooth hypersurface \( W_{0}^{(d-1)}\subset U \) satisfying \ref{EqContMaxProp}.
Set \( \mathcal{J}_{\xi_{0},0}^{(d-1)}= 
\Coeff_{W_{0}^{(d-1)}}(\mathcal{T}_{\xi_{0},0}) \).

If \( \mathcal{J}_{\xi_{0},0}^{(d-1)}=0 \) then we set \( 
\Fc_{0}^{(d-1)}=\infty \).

If \( \mathcal{J}_{\xi_{0},0}^{(d-1)}\neq 0 \) then,
by induction (\ref{HipotInducdim}),
the value \( \Fc_{0}^{(d-1)}(\xi_{0}) \) associated to 
\( \mathcal{J}_{\xi_{0},0}^{(d-1)} \) and \( 
(W_{0}^{(d-1)},E_{0}^{(d-1)}) \) is already defined, where we set
\( E_{0}^{(d-1)}=\emptyset \).

In any case, we may set
\begin{equation*}
    \Fc_{0}^{(d)}(\xi_{0})=
    (\omega_{\xi_{0},0}^{(d)},0,\Fc_{0}^{(d-1)}(\xi_{0}))
\end{equation*}
where the second component is \( 0\) because at the beginning there
are no exceptional divisors.  We will define this component in \ref{DefFuncN}.

The value \( \Fc_{0}^{(d)}(\xi_{0}) \) is well-defined and does not 
depend on the choice of the hypersurface \( V \). This follows from 
\ref{ThJarek} and property \ref{FuntorProp} applied to \( V \) and 
smooth morphisms of relative dimension zero.

With this procedure we define a function \( 
\Fc_{0}^{(d)}:W_{0}\to\Lambda \). The upper-semicontinuity follows by 
the upper-semicontinuity of functions \( \ord \) and \( \Fc_{0}^{(d-1)} \).

The closed set \( C_{0}=\MaxB\Fc_{0}^{(d)} \) is smooth by the 
inductive assumption on the dimension \( d-1 \).

On the other hand, \( C_{0} \) has only normal crossings with \( 
E_{0} \) by the definition of \( \mathcal{T}_{0} \).
\end{parrafo}

\begin{parrafo} \label{InducStep_i}
\textbf{Step i.}
Assume now that \( d=\dim{W}>1 \).  And suppose that
we have already constructed the first \( i \) steps of the sequence
\ref{SeqResolQRA}.
\begin{equation} \label{EqSeqResoli}
    (W_{0},E_{0})\longleftarrow (W_{1},E_{1})\longleftarrow \cdots 
    \longleftarrow (W_{i},E_{i})
\end{equation}
and the functions
\begin{equation*}
    \Fc_{0}^{(d)},\ \Fc_{1}^{(d)},\ldots,\ \Fc_{i-1}^{(d)}
\end{equation*}
satisfying properties in \ref{DefResolQRA} and \ref{PropSeqFc}, and 
also stability property \ref{FuntorProp} for smooth morphisms of 
relative dimension zero. 

For any \( j=0,\ldots,i-1 \),
the algebra \( \mathcal{J}_{j} \) is the transform of \( 
\mathcal{J}_{0} \) in \( W_{j} \).
We denote by \( E_{j,0} \) to be the set of strict transforms of \( E_{0} 
\) in \( W_{j} \). Note that \( E_{j,0}\subset E_{j} \).
We set the non-monomial part of \( \mathcal{J}_{j} \)
\begin{equation} \label{EqDefI}
    \mathcal{I}_{j}=(E_{j}\setminus E_{j,0})^{-1}\mathcal{J}_{j}
\end{equation}
For any point \( \xi_{i}\in W_{i} \) we denote \( \xi_{j}\in W_{j} \), 
\( j=0,\ldots,i-1 \) to be the image of \( \xi_{i} \) in \( W_{j} \).

By construction the first coordinate of the \( \Fc_{j}^{(d)}(\xi_{j}) 
\) is \( \omega_{\xi_{j},j}^{(d)}=\ord(\mathcal{I}_{j})(\xi_{j}) \).
We have the chain of inequalities
\begin{equation*}
    \omega_{\xi_{0},0}^{(d)}\geq \omega_{\xi_{1},1}^{(d)}\geq \cdots 
    \geq \omega_{\xi_{i-1},i-1}^{(d)}
\end{equation*}
\end{parrafo}

\begin{parrafo}
With the situation as in \ref{InducStep_i} we want to define the 
function \( \Fc_{i}^{(d)}:W_{i}\to\Lambda \).

Fix a point \( \xi_{i}\in W_{i} \).
Set \( I_{i}=(E_{i}\setminus E_{i,0})^{-1}\mathcal{J}_{i} \) to be the non-monomial 
part and set \( 
\omega_{\xi_{i},i}^{(d)}=\ord(\mathcal{I}_{i})(\xi_{i}) \). We have 
the inequalities
\begin{equation*}
    \omega_{\xi_{0},0}^{(d)}\geq \omega_{\xi_{1},1}^{(d)}\geq \cdots 
    \geq \omega_{\xi_{i-1},i-1}^{(d)}\geq \omega_{\xi_{i},i}^{(d)}
\end{equation*}
First note that if \( \omega_{\xi_{i},i}^{(d)}=0 \) then the algebra 
\( \mathcal{J}_{i} \) reduce to a monomial and it is easy to define 
\( \Fc_{i}^{(d)}(\xi_{i}) \) in order to enlarge \ref{EqSeqResoli} 
(locally at \( \xi_{i} \)) to a resolution.

So that we may assume \( \omega_{\xi_{i},i}^{(d)}>0 \).
Set \( j_{0} \) the minimum index such that
\begin{equation*}
    \omega_{\xi_{j_{0}},j_{0}}^{(d)}=\cdots
     = \omega_{\xi_{i-1},i-1}^{(d)}= \omega_{\xi_{i},i}^{(d)}
\end{equation*}
Denote by \( E_{i,j_{0}} \) to be the set of strict transforms of \( 
E_{j_{0}} 
\) in \( W_{i} \), \( E_{i,j_{0}}\subset E_{i} \).
Define
\begin{equation} \label{DefFuncN}
    n_{\xi_{i},i}=\# \{H\in E_{i,j_{0}}\mid \xi_{i}\in E_{i,j_{0}}\}
\end{equation}
If \( \omega=\omega_{\xi_{i},i}^{(d)} \) we set
\begin{equation*} \label{EqDefP_T}
    \mathcal{P}_{\xi_{i},i}=
    \mathcal{I}_{i}^{\frac{1}{\omega}}\odot \mathcal{J}_{i},
    \qquad
    \mathcal{T}_{\xi_{i},i}=
    \mathcal{P}_{x_{i},i}\odot \mathcal{E}_{\xi_{i},i}
\end{equation*}
where \( \mathcal{E}_{\xi_{i},i} \) is the \( \mathbb{Q} \)-Rees 
algebra generated by \( \{I(H)T\mid H\in E_{i,j_{0}}, \xi_{i}\in H\} \).

The order of \( \mathcal{P}_{\xi_{i},i} \) at \( \xi_{i} \) is one, 
so that there is an open neighborhood \( U \) where \( 
\mathcal{P}_{\xi_{i},i} \) is simple. By \ref{ThContMaxLocal} we 
could shrink \( U \) and choose a smooth hypersurface \( 
W_{i}^{(d-1)}\subset U \) such that \( I(W_{i}^{(d-1)})T\subset 
\Diff(\mathcal{P}_{\xi_{i},i})|_{U} \).

Set \( E_{i}^{(d-1)}=E_{i}\setminus E_{i,j_{0}} \) and 
\( \mathcal{J}_{i}^{(d-1)}= 
\Coeff_{W_{i}^{(d-1)}}(\mathcal{T}_{\xi_{i},i}) \).

If \( \mathcal{J}_{i}^{(d-1)}=0 \) then we set \( 
\Fc_{i}^{(d-1)}(\xi_{i})=\infty \).

If \( \mathcal{J}_{i}^{(d-1)}\neq 0 \) then by induction hypothesis 
we may consider the resolution of \( \mathcal{J}_{j_{0}}^{(d-1)} \) 
in \( W_{j_{0}}^{(d-1)} \). And we set
\begin{equation*}
    \Fc_{i}^{(d)}(\xi_{i})=(w_{\xi_{i},i}^{(d-1)}, n_{\xi_{i},i}, 
    \Fc_{i}^{(d-1)}(\xi_{i}))
\end{equation*}

The function \( \Fc_{i}^{(d)} \) is upper-semi-continuous by 
construction and the upper-semi-continuity of \( \Fc_{i}^{(d-1)} \).

In order to prove that the center \( C_{i}=\MaxB\Fc_{i}^{(d)} \) is 
smooth, note that if \( \Fc_{i}^{(d-1)}(C_{i})=\infty \) then \( 
C_{i}=W_{i}^{(d-1)} \). If \( \Fc_{i}^{(d-1)}(C_{i})\neq\infty \) 
then \( C_{i}=\MaxB\Fc_{i}^{(d-1)} \) and it is smooth by induction 
hypothesis.
\end{parrafo}

\begin{parrafo} \label{EquivSameResol}
Given a \( \mathbb{Q} \)-Rees algebra \( \mathcal{J} \) we have 
constructed a sequence \ref{SeqResolQRA} and functions \( 
\Fc_{i}^{(d)} \), \( i=0,\ldots,N-1 \) as in \ref{PropSeqFc}.

Assume that we have two equivalent \( \mathbb{Q} \)-Rees algebras \(
\mathcal{J} \) and \( \mathcal{J}' \) (\ref{DefIntEquivQRA}).  Then
associated to \( \mathcal{J} \) and \( \mathcal{J}' \) are two
sequences as in \ref{SeqResolQRA}, say:
\begin{gather} \label{TwoSeqEquiv1}
    (W,E)=(W_{0},E_{0})\longleftarrow (W_{1},E_{1})\longleftarrow 
    \cdots \longleftarrow (W_{N},E_{N}) \\
    \label{TwoSeqEquiv2}
    (W,E)=(W'_{0},E'_{0})\longleftarrow (W'_{1},E'_{1})\longleftarrow 
    \cdots \longleftarrow (W'_{N'},E'_{N'}) 
\end{gather}
We also have functions \(
\Fc_{0}^{(d)},\ldots,\Fc_{N}^{(d)} \), associated to \( \mathcal{J}
\), and \( {\Fc'_{0}}^{(d)},\ldots,{\Fc'_{N'}}^{(d)} \), associated to
\( \mathcal{J}' \).

We claim that both sequences are equal and moreover the functions are 
equal: \( N=N' \) and \( \Fc_{i}^{(d)}={\Fc'_{i}}^{(d)} \), \( 
i=0,\ldots,N-1 \).

To prove our claim is enough to see that all operations involved are 
stable by integral closure:

We may proceed by induction on \( i=0,\ldots,N-1 \). If \(
\mathcal{J}_{j} \) and \( \mathcal{J}'_{j} \), \( j=0,\ldots,i \) are
equivalent and the sequences in \ref{TwoSeqEquiv1} and
\ref{TwoSeqEquiv2} coincide for the first \( i \) steps, then
\begin{itemize}
    \item  the algebras \( \mathcal{I}_{j} \) and \( 
    \mathcal{I}'_{j} \) defined as in \ref{EqDefI} are equivalent,

    \item  the algebras \( \mathcal{P}_{\xi_{i},i} \), \( 
    \mathcal{T}_{\xi_{i},i} \) and \( \mathcal{P}'_{\xi_{i},i} \), \( 
    \mathcal{T}'_{\xi_{i},i} \) defined as in \ref{EqDefP_T} are, 
    respectively, equivalent.
\end{itemize}
Finally the centers \( C_{i}=C'_{i} \) coincide and the transforms \(
\mathcal{J}_{i+1} \) and \( \mathcal{J}'_{i+1} \) are equivalent.
\end{parrafo}

\def\cprime{$'$}
\providecommand{\bysame}{\leavevmode\hbox to3em{\hrulefill}\thinspace}
\providecommand{\MR}{\relax\ifhmode\unskip\space\fi MR }
\providecommand{\MRhref}[2]{%
  \href{http://www.ams.org/mathscinet-getitem?mr=#1}{#2}
}
\providecommand{\href}[2]{#2}

\end{document}